\documentclass[11pt]{article}
\usepackage{amssymb}
\usepackage{}
\usepackage{amsmath}
\usepackage{amsfonts}
\usepackage{latexsym,amsfonts,amsmath,amsthm,makeidx,amssymb}
\usepackage{mathrsfs}
\usepackage{makeidx}
\makeindex
\newtheorem{definition}{Definition}[section]
\newtheorem{theorem}{Theorem}[section]
\newtheorem{lemma}{Lemma}[section]
\newtheorem{corollary}{Corollary}[section]
\newtheorem{proposition}{Proposition}[section]
\newtheorem{remark}{Remark}[section]

\newcommand{\R}{\mathbb R}

\newcommand{\bpp}{\begin{proposition}}
\newcommand{\epp}{\end{proposition}}

\newcommand{\bt}{\begin{theorem}}
\newcommand{\et}{\end{theorem}}
\newcommand{\bl}{\begin{lemma}}
\newcommand{\el}{\end{lemma}}
\newcommand{\bd}{\begin{definition}}
\newcommand{\ed}{\end{definition}}
\newcommand{\bc}{\begin{corollary}}
\newcommand{\ec}{\end{corollary}}
\newcommand{\bp}{\begin{proof}}
\newcommand{\ep}{\end{proof}}
\newcommand{\bx}{\begin{example}}
\newcommand{\ex}{\end{example}}
\newcommand{\bi}{\begin{exercise}}
\newcommand{\ei}{\end{exercise}}
\newcommand{\bo}{\begin{prop}}
\newcommand{\eo}{\end{prop}}
\newcommand{\br}{\begin{remark}}
\newcommand{\er}{\end{remark}}
\newcommand{\be}{\begin{equation}}
\newcommand{\ee}{\end{equation}}
\newcommand{\ba}{\begin{align}}
\newcommand{\ea}{\end{align}}
\newcommand{\bn}{\begin{enumerate}}
\newcommand{\en}{\end{enumerate}}
\newcommand{\bg}{\begin{align*}}
\newcommand{\bcs}{\begin{cases}}
\newcommand{\ecs}{\end{cases}}

\newcommand{\bean}{\begin{eqnarray*}}
\newcommand{\eean}{\end{eqnarray*}}

\allowdisplaybreaks

\numberwithin{equation}{section}

\begin{document}
\title{\bf  Existence and asymptotic behavior for the ground state of quasilinear elliptic equation}
\date{}
\author{
{\bf Xiaoyu Zeng\quad Yimin Zhang }\thanks{ Corresponding author.
Supported by NSFC  under grant numbers 11471330 and 11501555.
E-mail: zengxy09@126.com; zhangym802@126.com} \\
{\small\it  Department of Mathematics, Wuhan University of Technology,
 }\\
{\small\it  Wuhan 430070, PR China}\\\\
}

\maketitle



\vskip0.36in

\begin{abstract}
 In this paper, we are concerned with the existence and asymptotic behavior of minimizers for a minimization problem related to  some
quasilinear elliptic equations. Firstly, we proved that there exist minimizers when the exponent $q$ equals to the critical case $q^*=2+\frac{4}{N}$, which is different from that  of \cite{cjs}. Then, we proved that all minimizers are compact as $q$ tends to the critical case $q^*$ when $a<a^*$  is fixed. Moreover, we studied the concentration behavior of minimizers as the exponent $q$ tends to the critical case $q^*$ for any fixed $a>a^*$.
 \end{abstract}


\vskip0.6in


\section{Introduction}
\indent\indent In this paper, we consider the following minimization problem
\begin{equation}\label{e1.1}
d_a(q)=\inf_{u\in M}E_q^a(u)
\end{equation}
where  $$M=\left\{\int_{\mathbb{R}^N}|u|^2dx=1,\ u\in
X\right\},$$
and
\begin{equation}\label{eq1.2}
E_q^a(u)=\frac{1}{2}\int_{\mathbb{R}^N} (|\nabla
u|^2+V(x)|u|^2)dx+\frac{1}{4}\int_{\mathbb{R}^N} |\nabla
u^2|^2dx-\frac{a}{q+2}\int_{\mathbb{R}^N} |u|^{q+2}dx.
\end{equation}
Here, we assume that  $0<q\leq q^*=2+\frac{4}{N}$, $a\in \mathbb{R}$ is a constant, the potential $V(x)\in L_{\rm loc}^\infty(\R^N; \R^+)$.  The space $X$ is defined by
$$X=\Big\{u: \int_{\mathbb{R}^N}|\nabla u^2|^2dx <\infty, u\in
H\Big\}$$
with
$$H=\Big\{u:  \int_{\mathbb{R}^N}|\nabla u|^2+V(x)|u|^2dx<\infty\Big\}.$$

Any minimizers of (\ref{e1.1}) solve the following quasilinear elliptic equation
\begin{equation}\label{1.1}
-\Delta u-\Delta
(u^2)u+V(x)u= \mu u+a|u|^{q}u,\
x\in\mathbb{R}^N.
\end{equation}
That is, the Euler-Lagrange equation to problem (\ref{e1.1}), where $\mu$ denotes the Lagrange multiplier under the constraint $\|u\|_{L^2}^2=1$.
Solutions of problem (\ref{1.1}) also correspond to the
 standing wave solutions of certain quasilinear Schr\"{o}dinger
equation
\begin{equation}\label{t1}
i\partial_t \varphi=-\Delta \varphi-\Delta
(\varphi^2)\varphi+W(x)\varphi- a|\varphi|^{q}\varphi,\
x\in\mathbb{R}^N,
\end{equation}
where $\varphi:\mathbb{R}\times \mathbb{R}^N\rightarrow \mathbb{C}$, $W:\mathbb{R}^N\rightarrow \mathbb{R}$ is a given potential. Equation (\ref{t1}) arising in several physical phenomena such as the
theory of plasma physics, exciton in one-dimensional lattices and dissipative quantum mechanics, see
for examples \cite{bl,sk,lss,lw} and the references therein for more
backgrounds. It is obvious that $e^{-i\mu t}u(x)$ solves (\ref{t1}) if and only if $u(x)$ is the solution of equation (\ref{1.1}).

Equation (\ref{1.1}) is usually called  a semilinear elliptic equation if we ignore the term $-\Delta
(u^2)u$. The  constrained minimization problem associated to semilinear elliptic equation has been studied widely \cite{bc,gs,GWZZ,gzz,GZZ}. The authors  considered in \cite{bc,gs,GWZZ,GZZ} the following minimization problem in dimension two:
\begin{equation}\label{e1.3}
I_a(q)=\inf_{u\in H,\ \int_{\mathbb{R}^2}|u|^2dx=1}J_q^a(u)
\end{equation}
where
\begin{equation*}
J_q^a(u)=\frac{1}{2}\int_{\mathbb{R}^2} |\nabla
u|^2dx+\frac{1}{2}\int_{\mathbb{R}^2}
V(x)|u|^2dx-\frac{a}{q+2}\int_{\mathbb{R}^2} |u|^{q+2}dx,
\end{equation*}
and $0<q\leq q^*=2$, $a\in \mathbb{R}$ is a constant.
By using some rescaling arguments, they  obtained that there exists a constant $a^*$, such that (\ref{e1.3}) has at least one minimizer if and only if $a<a^*$. Moreover, it was discussed in  \cite{gs,GWZZ,GZZ} further the concentration and symmetry breaking of minimizers  for (\ref{e1.3}) when $q=2$ and $a$ tends to $a^*$ from left (denoted by $a\nearrow a^*$). Recently, Guo, Zeng and Zhou \cite{gzz} studied the concentration behavior of minimizers of (\ref{e1.3}) as $q\nearrow2$ for any fixed $a>a^*$.

There are amount of work considering the existence of solutions for equation (\ref{1.1}), see \cite{cj,cjs,lw,lww3,lww4} for subcritical case and \cite{dpy,jms,llw,zzz,zws} for critical case. By using a constrained minimization argument,  for different types of potentials the authors established in \cite{lw,lww4} the existence of positive solutions of problem (\ref{1.1}) on the manifold
$M=\left\{\int_{\mathbb{R}^N}|u|^{q+2}=c,\ u\in X\right\}$ and Nahari manifold  when $2\leq q<\frac{2(N+2)}{N-2}$.
 In \cite{cj,lww3}, by changing  of variables,  (\ref{1.1}) was transformed to a semilinear elliptic equation, then the existence of positive solutions were obtained by mountain
pass theorem in Orlitz space  or Hilbert space framework. It is  worth mentioning that the authors in  \cite{cjs,jl} investigated
the follwing constrained minimization problem associated to the quasilinear
elliptic equation (\ref{1.1}) with $V(x)=\rm constant$:
 \begin{equation}\label{e1}
m(c)=\inf\{E(u): |u|_{L^2}^2=c\},
\end{equation}
where
\begin{equation}\label{eq1.60}
E(u)=\frac{1}{2}\int_{\mathbb{R}^N} |\nabla
u|^2dx+\frac{1}{4}\int_{\mathbb{R}^N} |\nabla
u^2|^2dx-\frac{1}{q+2}\int_{\mathbb{R}^N} |u|^{q+2}dx.
\end{equation}
They mainly obtained  that for any $c>0$, then
 $$m(c)=\begin{cases}-\infty, &\text{ if }q>q^*=2+\frac{4}{N},\\
  0, &\text{ if }q=q^*=2+\frac{4}{N},\end{cases}$$
 and  (\ref{e1}) possesses no minimizer.   On the other hand, when $q\in (0, 2+\frac{4}{N})$ there holds that $m(c)\in(-\infty,0]$. Especially, if  the energy is strictly less than zero, namely,
\begin{equation}\label{eq1.6}
m(c)\in(-\infty,0),
\end{equation}
 they proved that (\ref{e1}) possesses at leat one  minimizer   by using Lions' concentration-compactness principle. In general, condition (\ref{eq1.6}) can be verified  for any $c>0$ if $q\in(0,\frac{4}{N})$. But for the case of
$q\in(\frac{4}{N},2+\frac{4}{N})$, by setting
$$c(q,N):=\inf\{c>0:m(c)<0\},$$
 it was proved in \cite{jl} that $c(q,N)>0$ and (\ref{e1}) is achieved if and only if  $c\in[c(q,N),+\infty)$.
  Based on the above results, Jeanjean, Luo and Wang \cite{jlw} recently  discovered   that there exists $\hat{c}\in(0, c(q,N))$, such that functional (\ref{eq1.60}) admits  a local minimum on the manifold $\{u\in X: |u|_{L^2}^2=c\}$ for all $c\in(\hat c, c(q,N))$ and $q\in(\frac{4}{N},2+\frac{4}{N})$.  Furthermore, mountain pass type critical point of (\ref{eq1.60}) was also obtained therein for all $c\in(\hat c, \infty)$, which is different from the minimum solution.

  We note that, by taking the scaling $u_c(x)=u(c^\frac{1}{N}x)$, then
  $$E(u)=c^{1-\frac{2}{N}}\Big\{\frac{1}{2}\int_{\mathbb{R}^N} |\nabla
u_c|^2dx+\frac{1}{4}\int_{\mathbb{R}^N} |\nabla
u_c^2|^2dx-\frac{c^{\frac{2}{N}}}{q+2}\int_{\mathbb{R}^N} |u_c|^{q+2}dx\Big\}.$$
It would be easy to see that problem (\ref{e1}) can be equivalently transformed to  problem (\ref{e1.1}) with $V(x)\equiv {\rm constant}$ (Without loss of generality, we assume $V(x)\equiv0$) by setting $
  a=c^{\frac{2}{N}},
$ namely, the following minimization problem:
\begin{equation}\label{eq1.3}
\tilde{d}_a(q)=\inf_{u\in M}\tilde{E}_q^a(u),
\end{equation}
where $\tilde{E}_q^a(\cdot)$ is given by
\begin{equation*}
\tilde{E}_q^a(u)=\frac{1}{2}\int_{\mathbb{R}^N} |\nabla
u|^2dx+\frac{1}{4}\int_{\mathbb{R}^N} |\nabla
u^2|^2dx-\frac{a}{q+2}\int_{\mathbb{R}^N} |u|^{q+2}dx.
\end{equation*}
  From the above known results of minimization problem (\ref{e1}), we see that  (\ref{eq1.3}) could be achieved  only if $q<q^*=2+\frac{4}{N}$. The exponent $q^*=2+\frac{4}{N}$ seems to be  the critical exponent for the existence of minimizers for (\ref{eq1.3}). 
  A natural question one would ask  is that does problem (\ref{e1.1}) admit minimizers if $V(x)\not\equiv {\rm constant}$ ? Taking the scaling $u^\sigma(x)=\sigma^{\frac{N}{2}}u(\sigma x)$, it is easy to know that
$E_q^a(u^\sigma)\rightarrow -\infty$ as $\sigma\rightarrow +\infty$
if $q>q^*$ and $V(x)\in L_{\rm loc}^\infty(\R^N)$. This implies that there is no minimizer for
problem (\ref{e1.1}) when $q>q^*$. However, when $q=q^*$, the result is quite different. Indeed, for a class  of non-constant potentials, we will prove  that there exists a threshold (w.r.t. the parameter $a$)  independent of $V(x)$ for the existence of minimizers for (\ref{e1.1}), see our Theorem \ref{t1.1} below  for  details.  Moreover, stimulated by \cite{gzz},  we are further interested in studying the limit behavior of minimizers for (\ref{e1.1}) as   $q\nearrow q^*$.

 Before stating our main results, we first recall the following  sharp Gagliardo-Nirenberg
inequality  \cite{ma}:
\begin{equation}\label{1.5}
\int_{\R^N}|u|^\frac{q+2}{2}dx\leq \frac{1}{\Upsilon_q}\Big(\int_{\R^N}|\nabla
u|^2dx\Big)^{\frac{(q+2)\theta_q}{4}}|u|_{L^1}^{\frac{(q+2)(1-\theta_q)}{2}}, \forall \ u\in \mathcal {D}^{2,1}(\mathbb{R}^N),
\end{equation}
where $1<\frac{q+2}{2}<\frac{2N}{N-2}$, $\theta_q=\frac{2qN}{(q+2)(N+2)}$, $\Upsilon_q>0$ and
$$\mathcal {D}^{2,1}(\mathbb{R}^N)\triangleq\big\{u: \nabla u\in L^2(\R^N), u\in L^1(\R^N)\big\}.$$
As proved in \cite{ma}, the optimal constant $\Upsilon_q=\lambda_q a_q$ with
\begin{equation}\label{eq1.12}
\lambda_q=(1-\theta_q)\left(\frac{\theta_q}{1-\theta_q}\right)^{\frac{qN}{2(N+2)}} \text{ and }\ a_q=|v_q|_{L^1}^{\frac{q}{N+2}}.
\end{equation}
Here, $ v_q\geq0 $ optimizes (\ref{1.5}) (that is, (\ref{1.5}) is an identity if $u=v_q$) and is the unique nonnegative radially symmetric  solution  of the following equation \cite{st}  \begin{equation}\label{1.3}
-\triangle v_q+1=v_q^{\frac{q}{2}},\ \ x\in \mathbb{R}^N.
\end{equation}
\begin{remark}
Strictly speaking, it has been proved in \cite[Theorem 1.3 (iii)]{st} that $v_q$ has a compact support in $\R^N$ and exactly  satisfies a  Dirichlet-Neumann free boundary problem. Namely, there exists one $R>0$ such that $v_q$ is the unique positive solution of
\begin{equation}\label{eq1.144}
\begin{split}
&-\Delta u+1=u^{\frac{q}{2}},\\
& u>0 \text{ in } B_R, u=\frac{\partial u}{\partial n}=0 \text{ on } \partial B_R.
\end{split}
\end{equation}
In what follows, if we say that $u$ is a nonnegative solution of a equation like the form of (\ref{1.3}),  we exactly means that $u$ is a solution of the free boundary problem (\ref{eq1.144}).
\end{remark}

From  equation (\ref{1.3}) and the classical Pohozaev identity, one can prove that
\begin{equation} \label{1.4}
\left\{ \begin{aligned}
         \int_{\mathbb{R}^N} |v_q|^\frac{q+2}{2}dx &=\frac{1}{1-\theta_q} \int_{\mathbb{R}^N} |v_q|dx, \\
                  \int_{\mathbb{R}^N} |\nabla
v_q|^2dx&=\frac{\theta_q}{1-\theta_q}\int_{\mathbb{R}^N} |v_q|dx.
                          \end{aligned} \right.
                          \end{equation}
 Using the above notations, we first obtain  the following result which addresses the existence of minimizers about problem (\ref{e1.1}) for the critical case of $q=q^*$.

\begin{theorem}\label{t1.1}
 Let $q=q^*=2+\frac{4}{N}$ and $a_{q^*}$ be given by (\ref{eq1.12}). Assume that $V(x)$ satisfies
\begin{equation}\label{v}
V(x)\in L^\infty_{\rm loc}(\R^N;\R^+), \  \ \inf_{x\in\R^N}V(x)=0\ \text{ and }\ \lim_{|x|\to\infty}V(x)=\infty.
\end{equation}
Then,
\begin{itemize}
\item [\rm(i)] $d_a(q^*)$ has at least one minimizer if $0<a\le a_{q^*}$,
\item [\rm (ii)] There is no minimizer for $d_a(q^*)$ if $a>a_{q^*}$.
\end{itemize}
\end{theorem}

Theorem \ref{t1.1} is mainly stimulated by \cite[Theorem 2.1]{bc} and \cite[Theorem 1]{gs}, where the semilinear  minimization problem (\ref{e1.3}) was studied. The argument in these two references for studying  the non-critical case, namely $a\not= a_{q^*}$ is useful for solving our problem.  However, when $a$ equals to the threshold (i.e., $a=a^*$ in their problem), it was proved in \cite{bc,gs} that  there is no minimizer for  problem (\ref{e1.3}). This is quite different from our case since there exists at least one minimizer for (\ref{e1.1}) when $a=a_{q^*}$. This difference is mainly caused by the presence of the extra term $\int_{\R^N}|\nabla u|^2dx$ in (\ref{eq1.2}), which makes the argument in \cite{bc,gs} unavailable  for studying  our problem.  To deal with the critical case,  we  will introduce in Section \ref{se2}  a suitable auxiliary functional and  obtain the boundness of minimizing sequence by contradiction. Then, the existence of minimizers follows directly from the compactness Lemma \ref{le1}.



We  remark that if  $V(x)$ satisfies condition (\ref{v}), one can easily apply the  Gagliardo-Nirenberg inequality (\ref{1.5}) and Lemma \ref{le1}  to  prove that (\ref{e1.1}) possesses minimizers for any fixed $1<q<q^*$. In what follows,  we investigate  the limit behavior of minimizers for (\ref{e1.1}) as $q\nearrow q^*$. Firstly, if $a<a_{q^*}$ is fixed,  our result shows that  the minimizers of (\ref{e1.1}) is compact in the space $X$ as $q\nearrow q^*$. More precisely, we have
\begin{theorem}\label{th1.2}
 Assume $V(x)$ satisfies (\ref{v}) and let $u_q\in M$ be a nonnegative minimizer of problem (\ref{e1.1}) with  $0<a<a_{q^*}$ and $0<q<q^*=2+\frac{4}{N}$. Then
\begin{equation*}
\lim_{q\nearrow q^*}d_a(q)=d_a(q^*).
\end{equation*}
Moreover, there exists $u_0\in M$ such that  $\lim_{q\nearrow q^*}u_q= u_0$ in $X$, where $u_0$ is a nonnegative minimizer of $d_a(q^*)$. Here, the sequence  $\lim_{q\nearrow q^*}u_q= u_0$ in $X$ means that $$u_q \to u_0 \text{ in } H \text{ and  } \int_{\R^N}|\nabla u_q^2|^2dx\to\int_{\R^N}|\nabla u_0^2|^2dx \ \  \text{as } q\nearrow q^*.$$
\end{theorem}
On the contrary, if $a>a_{q^*}$,  the result is quit different and blow-up will happen in minimizers as $q\nearrow q^*$. Actually, our following theorem tells  that all minimizers of (\ref{e1.1}) must concentrate and blow up at one minimal point of the potentials.
\begin{theorem}\label{th1.3}
Assume V(x) satisfies (\ref{v}) and $a>a_{q^*}$ . Let $\bar u_q$ be a non-negative minimizer of (\ref{e1.1}) with $0<q<q^*$. For any sequence of $\{\bar u_q\}$, by passing to subsequence if necessary, then there exists $\{y_{\varepsilon_q}\}\subset \R^N$ and $y_0\in\R^N$ such that
\begin{equation*}
\lim_{q\nearrow q^*} \varepsilon_q^{N}\bar u_q^2(\varepsilon_q
x+\varepsilon_qy_{\varepsilon_q})=\frac{\lambda^N}{|v_{q^*}|_{L^1}}v_{q^*}\Big(\lambda |x-y_0|\Big) \text{ strongly in } \mathcal {D}^{2,1}(\mathbb{R}^N),
\end{equation*}
where $v_{q^*}$ is the unique nonnegative radially symmetric  solution  of (\ref{1.3}) and
\begin{equation}\label{eq1.15}\lambda=\left(\frac{|v_{q^*}|_{L^1}}{N}\right)^\frac{1}{N+2},\ \varepsilon_q=\left(\frac{4aq}{q^*\lambda_q
a_q(q+2)}\right)^{-\frac{2}{N(q*-q)}}\overset{q\nearrow q^*}\longrightarrow0^+.\end{equation}
Moreover, taking $A:=\{x:V(x)=0\} $, then the sequence $\{y_{\varepsilon_q}\}$ satisfies
$$\text{dist} \big(\varepsilon_q y_{\varepsilon_q}, A\big)\to0 \ \text{ as }\ q\nearrow q^*.$$


\end{theorem}

%
%
%

Throughout the paper, $|u|_{L^p}$ denotes the $L^p$-norm of
function $u$, $C, c_0, c_1$ denote some constants.

This paper is organized as follows.  In Section \ref{se2} we shall prove Theorem \ref{t1.1} by some rescaling arguments, especially, we prove that $d_{a^*}(q^*)$ possesses minimizers by  introducing an auxiliary minimization problem.    Section \ref{se3} is devoted to the proof of Theorem \ref{th1.2} on the  compactness in space $X$ for minimizers of $d_a(q^*)$ as $q\nearrow q^*$. In Section \ref{se4},  we  first establish optimal energy estimates for $d_a(q)$  as $q\nearrow q^*$ for any fixed $a>a_{q^*}$, upon which we then  complete the proof of Theorem \ref{th1.3} on the  concentration behavior of nonnegative minimizers as $q\nearrow q^*$.

\section{The existence of minimizers: Proof of Theorem \ref{t1.1}.}\label{se2}
 The main purpose of this section  is to establish Theorem \ref{t1.1}. We first introduce the following lemma, which was  essentially proved in  \cite[Theorem XIII.67]{RS} and \cite[Theorem 2.1]{bw}, etc.
\begin{lemma}\label{le1}
Assume $V(x)$ satisfies (\ref{v}), then the embedding from $H$ into $L^p(\mathbb{R}^N)$ is compact for all $2\leq p<2^*=\begin{cases}+\infty, \ \  N=1,2,\\
\frac{2N}{N-2}, \ \ N\geq3.
\end{cases}$
\end{lemma}
Taking $q=q^*=2+\frac{4}{N}$ in (\ref{1.3}), we get
$\theta_{q^*}=\frac{N}{N+1}$, $\lambda_{q^*}=\frac{N}{N+1}$ and
$a_{q^*}=|v_{q^*}|_{L^1}^{\frac{2}{N}}$. Moreover,
(\ref{1.4}) becomes
\begin{equation} \label{2.1}
\left\{ \begin{aligned}
         \int_{\mathbb{R}^N} |v_{q*}|^{\frac{q^*+2}{2}}dx &=(N+1) \int_{\mathbb{R}^N} |v_{q*}|dx, \\
                  \int_{\mathbb{R}^N} |\nabla
v_{q*}|^2dx&=N\int_{\mathbb{R}^N} |v_{q*}|dx,
                          \end{aligned} \right.
                          \end{equation}
and the Gagliardo-Nirenberg
inequality (\ref{1.5}) can be simply given  as
\begin{equation}\label{2.2}
\int_{\mathbb{R}^N}|u|^{\frac{q^*+2}{2}}dx\leq
\frac{N+1}{Na_{q*}}\int_{\mathbb{R}^N}|\nabla u|^2dx\cdot |u|_{L^1}^{\frac{2}{N}}.
\end{equation}

Inspired by the argument of \cite{bc,gs}, we first prove the following lemma which addresses Theorem \ref{t1.1} for the case of $a\not= a_{q^*}$.
\begin{lemma}\label{exi1} Let $V(x)$ satisfy (\ref{v}) and $q=q^*$. Then
\begin{itemize}
\item [\rm{(i)}] $d_a(q^*)$ has at least one minimizer if $0<a<a_{q^*}$;
\item [\rm{(ii)}] There is no minimizer for $d_a(q^*)$ if $a>a_{q^*}$.
\end{itemize}
\end{lemma}

\begin{proof}

 \textbf{(i)} If $a<a_{q^*}$,  for  any $  u\in M$,  it follows from (\ref{2.2}) that
\begin{align*}
\int_{\mathbb{R}^N}|u|^{q^*+2}dx
\leq \frac{N+1}{Na_{q^*}}\int_{\mathbb{R}^N} |\nabla u^2|^2dx \left(\int_{\mathbb{R}^N}u^2dx\right)^{\frac{2}{N}}\leq \frac{N+1}{Na_{q^*}}\int_{\mathbb{R}^N} |\nabla u^2|^2dx.
\end{align*}
Thus,
\begin{align}
E_{q^*}^a(u)&=\frac{1}{2}\int_{\mathbb{R}^N} \big(|\nabla
u|^2 +V(x)|u|^2\big)dx+\frac{1}{4}\int_{\mathbb{R}^N} |\nabla
u^2|^2dx-\frac{a}{q+2}\int_{\mathbb{R}^N} |u|^{q^*+2}dx\nonumber\\
&\geq \frac{1}{2}\int_{\mathbb{R}^N} \big(|\nabla
u|^2 +V(x)|u|^2\big)dx+\frac{1}{4}\left(1-\frac{a}{a_{q^*}}\right)\int_{\mathbb{R}^N} |\nabla
u^2|^2dx.\label{eq2.2}
\end{align}
Hence, if $\{u_n\}$ is a minimizing sequence of $d_a(q^*)$ with
$a<a_{q^*}$, it is easy to know from above  that there exists $C>0$ independent of $n$ such that
\begin{equation*}
\sup_n\int_{\mathbb{R}^N} |\nabla
u_n^2|^2dx\le C<\infty,\  \ \sup_n\|u_n\|_{H}\le C<\infty.
\end{equation*}
 It then follows from Lemma \ref{le1} that  there exists  a subsequence of $\{u_n\}$, denoted still by $\{u_n\}$, and $u\in M$ such that
\begin{equation}\label{eq2.3}
u_n\overset{n}\rightarrow u\ \text{ in }\ L^p(\mathbb{R}^N),\ \forall\ 2\leq p<2^*\end{equation}
and
$$\int_{\mathbb{R}^N} |\nabla
u^2|^2dx\leq \liminf_{n\rightarrow \infty}\int_{\mathbb{R}^N} |\nabla
u_n^2|^2dx\le C<\infty.$$
The latter inequality indicates  $u_n^2$ is bounded in $L^{2^*}(\mathbb{R}^N)$. Hence we can deduce from (\ref{eq2.3}) that
$$u_n\overset{n}\rightarrow u\ \text{ in }\ L^p(\mathbb{R}^N),\ \forall\  2\leq p<2\times 2^*.$$
Therefore,
$$d_a(q^*)= \liminf_{n\rightarrow \infty}E_{q^*}^a(u_n)\geq E_{q^*}^a(u)\geq d_a(q^*).$$
This indicate that
$$u_n \overset{n}\rightarrow u\ \text{ in }\ H\ \ \text{ and }\ \ \lim_{n\rightarrow \infty}\int_{\mathbb{R}^N} |\nabla
u_n^2|^2dx= \int_{\mathbb{R}^N} |\nabla
u^2|^2dx.$$
It is means that $u$ is a minimizer of $d_a(q^*)$ for $a<a_{q^*}$.

\textbf{(ii)} Let
$$
u_\tau=\frac{\tau^{\frac{N}{2}}}{|v_{q^*}|_{L^1}^{\frac{1}{2}}}\sqrt{v_{q^*}(\tau|x-x_0|)} \ \text{ with some } x_0\in \R^N.
$$
Using (\ref{2.1}) we have
\begin{equation}\label{2.3}
\int_{\mathbb{R}^N}u_\tau^2dx=\frac{\tau^N}{|v_{q^*}|_{L^1}}\int_{\mathbb{R}^N}|v_{q^*}(\tau x)|dx=\frac{1}{|v_{q^*}|_{L^1}}\int_{\mathbb{R}^N}|v_{q^*}|dx=1,
\end{equation}
\begin{equation}\label{2.4}
\int_{\mathbb{R}^N}|\nabla u_\tau^2|^2dx=\frac{\tau^{N+2}}{|v_{q^*}|_{L^1}^2}\int_{\mathbb{R}^N}|\nabla v_{q^*}|^2dx=\frac{N\tau^{N+2}}{|v_{q^*}|_{L^1}},
\end{equation}
\begin{equation}\label{2.5}
\int_{\mathbb{R}^N}u_\tau^{q^*+2}dx=\frac{\tau^{N+2}}{|v_{q^*}|_{L^1}^{2+\frac{2}{N}}}\int_{\mathbb{R}^N}v_{q^*}^{\frac{2(N+1)}{N}}dx=\frac{(N+1)\tau^{N+2}}{|v_{q^*}|_{L^1}^{1+\frac{2}{N}}},
\end{equation}
\begin{equation}\label{2.7}
\int_{\mathbb{R}^N}|\nabla u_\tau|^2dx=\frac{\tau^{2}}{|v_{q^*}|_{L^1}}\int_{\mathbb{R}^N}|\nabla \sqrt{v_{q^*}}|^2dx=c_0\tau^2,
\end{equation}
and
\begin{align}
\int_{\mathbb{R}^N}V(x)u_\tau^2dx&=\frac{\tau^N}{|v_{q^*}|_{L^1}}\int_{\mathbb{R}^N}V(x)|v_{q^*}(\tau x)|dx=\frac{1}{|v_{q^*}|_{L^1}}\int_{\mathbb{R}^N}V(\frac{x}{\tau}+x_0)|v_{q^*}|dx\nonumber\\
&\to V(x_0)\  \text{ as } \tau \to+\infty.\label{2.6}
\end{align}
If $a>a_{q^*}$, from (\ref{2.4}) and (\ref{2.5}) we have
\begin{equation*}
\frac{1}{4}\int_{\mathbb{R}^N}|\nabla
u_\tau^2|^2dx-\frac{a}{q^*+2}\int_{\mathbb{R}^N}u_\tau^{q^*+2}dx=\frac{N\tau^{N+2}}{4|v_{q^*}|_{L^1}}\left(1-\frac{a}{a_{q^*}}\right).
\end{equation*}
This together with (\ref{2.7}) and  (\ref{2.6}) gives that
$$d_a(q^*)\leq V(x_0)+o(1)+\frac{N\tau^{N+2}}{4|v_{q^*}|_{L^1}}\left(1-\frac{a}{a_{q^*}}\right)+c_0\tau^2\rightarrow -\infty\ \text{ as } \ \tau\to+\infty.$$
Therefore, we deduce that   $d_a(q^*)=-\infty$ and possesses  no minimizer  if $a>a_{q^*}$.
\end{proof}

In view of the above lemma, to complete the proof of Theorem \ref{t1.1}, it remains to deal with the case of  $a=a_{q^*}$.  In the following lemma, we first prove there exist minimizers for $d_{a_{q^*}}(q^*)$ when   $N\le3$.

\begin{lemma}\label{exl3}
 Assume that  $V(x)$ satisfies (\ref{v}),  then $d_a(q^*)$ has at least one minimizer if $a=a_{q^*}$ and $N\leq 3$.
\end{lemma}
\begin{proof}
Assume $\{u_n\}$ is a minimizing sequence of $d_{a_{q^*}}(q^*)$, similar to the argument of (\ref{eq2.2}), it is easy to know that
$$
\sup_n\|u_n\|_{H}\le C<+\infty.$$
Note that $q^*=2+\frac{4}{N}<2^*-2$ in view of $N\leq 3$, we thus deduce from the above inequality  that
$$\sup_{n}\int_{\mathbb{R}^N}|u_n|^{q^*+2}dx\le C<+\infty.$$
This further indicates that
$$\sup_n\int_{\mathbb{R}^N} |\nabla
u_n^2|^2dx\le C<+\infty.$$
Then similar to the proof (i) of Lemma \ref{exi1}, we know that there exists  a minimizer of $d_a(q^*)$ and the  proof is complete.
\end{proof}

When $N\geq 4$, the argument of Lemma \ref{exl3} become invalid to obtain the existence of minimizers for $d_{a_{q^*}}(q^*)$ since there holds that $q^*>2^*-2$. To deal with this case, we  introduce the following auxiliary minimization problem
\begin{equation}\label{e2.8}
m(c)=\inf\{F(u);\int_{\mathbb{R}^N}|u|dx=c,\ u\in \mathcal{D}^{2,1}(\mathbb{R}^N)\},
\end{equation}
where
$$F(u)=\int_{\mathbb{R}^N}|\nabla u|^2dx-\frac{Na_{q^*}}{N+1}\int_{\mathbb{R}^N}|u|^{2+\frac{2}{N}}dx.$$

\begin{lemma}\label{lem2.4}
\begin{itemize}
\item [\rm{(i)}] $m(c)=0$ if $0<c\leq 1$; $m(c)=-\infty$ if $c>1$.
\item [\rm{(ii)}] problem (\ref{e2.8}) possesses a minimizer if and only if $c=1$ and all nonegative minimizer of $m(1)$ must be of the form
\begin{equation}\label{e2.9}
\left\{\frac{\lambda^Nv_{q^*}(\lambda |x-x_0|)}{|v_{q^*}|_{L^1}}:\ \lambda\in \mathbb{R}^+, x_0\in\R^N\right\}.
\end{equation}
\end{itemize}
\end{lemma}

\begin{proof}
\textbf{(i)} follows easily by some scaling arguments.

\textbf{(ii).} From (i), we have $m(c)=0$ for any $c\leq 1$ . Thus, if $u$ is a minimizer of $m(c)$ with $c<1$,  we obtain from inequality (\ref{2.2}) that
$$0=m(c)\geq \left(1-c^{\frac{2}{N}}\right)\int_{\mathbb{R}^N}|\nabla u|^2dx.$$
This implies that $u\equiv 0$,  it is a contradiction.

If $c=1$, one can easily check that any $u_0$ satisfying (\ref{e2.9}) is a minimizer of $m(1)$. On the other hand, if $u\geq 0$ is a nonnegative minimizer of $m(1)$, we get that
$$\int_{\mathbb{R}^N}|\nabla u|^2dx=\frac{Na_{q^*}}{N+1}\int_{\mathbb{R}^N}|u|^{2+\frac{2}{N}}dx,$$
which indicates that $u$ is an optimizer of Gagliardo-Nirenberg inequality (\ref{2.2}). Hence, $u$ must be the form of (\ref{e2.9}).
\end{proof}

\begin{lemma}\label{l2.5}
Let $\{u_n\}\subset \mathcal{D}^{2,1}(\mathbb{R}^N)$ be a nonnegative minimizing sequence of $m(1)$, and $\int_{\mathbb{R}^N}|\nabla u_n|^2dx=1$. Then, there exists $\{y_n\}\subset \R^N$ and $x_0\in\R^N$, such that
$$\lim_{n\to\infty}u_n(x+y_n)= \frac{\lambda_0^N}{|v_{q^*}|_{L^1}}v_{q^*}(\lambda_0|x-x_0|) \ \text{ in }\ \mathcal{D}^{2,1}(\mathbb{R}^N)\ \text{ with }\ \lambda_0=\left(\frac{|v_{q^*}|_{L^1}}{N}\right)^{\frac{1}{N+2}}.$$
\end{lemma}

\begin{proof}
From the definition of $m(1)$ and $\{u_n\}$, one has
\begin{equation}\label{e2.10}
\int_{\mathbb{R}^N}|u_n|^{2+\frac{2}{N}}dx=\frac{N+1}{Na_{q^*}}+o(1).
\end{equation}
We will prove the compactness of $\{u_n\}$ by Lion's concentration-compactness principle \cite{lions1,lions2}.

(I) we first rule out the possibility of {\em vanishing:} If for any $R>0$,
$$\limsup_{y\in\mathbb{R}^N}\int_{B_R(y)} |u_n|dx=0.$$
Then $u_n\overset{n}\rightarrow 0$ in $L^q(\mathbb{R}^N)$ for any $1< q<2^*$, this contradicts  (\ref{e2.10}).

(II) Now, if {\em dichotomy} occurs, i.e., for some $c_1\in(0,1)$, there exist  $R_0, R_n>0$,  $\{y_n\}\subset \mathbb{R}^N$ and sequences $\{u_{1n}\}$, $\{u_{2n}\}$ such that
$$\text{supp }u_{1n}\subset B_{R_0}(y_n),\  \ \text{ supp } u_{2n}\subset B_{R_n}^c(y_n),$$
$$\text{dist}(\text{supp } u_{1n}, \text{ supp } u_{2n})\rightarrow+\infty\ \text{ as }\ n\rightarrow \infty,$$
$$\left|\int_{\mathbb{R}^N}|u_{1n}|dx-c_1\right|\leq \varepsilon,\ \left|\int_{\mathbb{R}^N}|u_{2n}|dx-(1-c_1)\right|\leq \varepsilon,$$
$$\liminf_{n\rightarrow \infty}\int_{\mathbb{R}^N}|\nabla u|^2dx\geq \liminf_{n\rightarrow \infty}\left(\int_{\mathbb{R}^N}|\nabla u_{1n}|^2dx+\int_{\mathbb{R}^N}|\nabla u_{2n}|^2dx\right)\geq 0.$$
Let $\tilde{u}_{1n}(x)=u_{1n}(x+y_n)$, then there is $u\in \mathcal{D}^{2,1}(\mathbb{R}^N)$, such that
\begin{equation}\label{eq2.13}\tilde{u}_{1n}(x)\overset{n}\rightarrow u\not\equiv 0\ in\ L_{loc}^p(\mathbb{R}^N),\ \ \forall\  1\leq p<2^*.\end{equation}
Since
$$0=m(1)=F(u_{1n})+F(u_{2n})+o_n(1)+\alpha(\varepsilon),$$
where $\alpha(\varepsilon)\to 0$ as $\varepsilon\to0$.
Letting $n\to\infty$ and then $\varepsilon\to0$, we then obtain from (\ref{eq2.13}) that
$$\int_{\R^N}|u|dx=c_1<1 \ \text{ and } \ 0\leq F(u)\leq m(1)=0.$$
This together with (\ref{2.2})  implies that $u=0$, it is a contradiction.

(III) The above discussions indicate that {\em compactness } occurs, i.e., for any $\varepsilon>0$, $\exists R>0$ and  $\{y_n\}\subset \mathbb{R}^N$, such that if $n$ is large enough,
$$\int_{B_R(y_n)} |u_n|dx\geq 1-\varepsilon.$$
Then, there exists  $u\in \mathcal{D}^{2,1}(\mathbb{R}^N)$ such that
$$u_n(x+y_n)\rightarrow u\ \text{ in }\ L^1(\mathbb{R}^N).$$
Therefore,we have $\int_{\mathbb{R}^N}|u|dx=1$ and
$$m(1)=\lim_{n\rightarrow\infty}F(u_n)\geq F(u)\geq m(1),$$
This indicates that  $u\geq 0$ is a minimizer of $m(1)$ and
\begin{equation}\label{e2.11}
\int_{\mathbb{R}^N}|\nabla u|^2dx=\lim_{n\rightarrow\infty}\int_{\mathbb{R}^N}|\nabla u_n|^2dx=1.
\end{equation}
Moreover, we obtain from Lemma \ref{lem2.4} (ii) that
$$u(x)=\frac{\lambda^N}{|v_{q^*}|_{L^1}}v_{q^*}(\lambda |x-x_0|),$$
where $\lambda=\left(\frac{|v_{q^*}|_{L^1}}{N}\right)^{\frac{1}{N+2}}$  follows directly from (\ref{e2.11}).  This finish the proof of the lemma.
\end{proof}

Based on the two above lemmas, we are ready to prove that when $a=a_{q^*}$,  there exist minimizers for $d_{a}(q^*)$  for any dimension $N\ge1$.

\begin{lemma}\label{ex3}Assume that $V(x)$ satisfies (\ref{v}),
then $d_{a_{q^*}}(q^*)$ has at least one minimizer if $a=a_{q^*}$.
\end{lemma}
\begin{proof}
Let $\{u_n\}$ be a minimizing sequence of $d_{a_{q^*}}(q^*)$. Similar to (\ref{eq2.2}), one can deduce from (\ref{2.2}) that  $\{u_n\}$ is  bounded in $H$.
Now, we claim that \begin{equation}\label{eq2.14}\text{$\int_{\mathbb{R}^N}|\nabla u_n^2|^2dx$ is also bounded uniformly as } \ n\to\infty. \end{equation} Otherwise, if
$$\lim_{n\rightarrow \infty}\int_{\mathbb{R}^N}|\nabla u_n^2|^2dx=+\infty.$$
Since $\{u_n\}$ is a minimizing sequence of $d_{a_{q^*}}(q^*)$, it then follows from above that
$$\lim_{n\rightarrow \infty}\int_{\mathbb{R}^N}|u_n|^{4+\frac{4}{N}}dx =+\infty
\ \text{ and }\lim_{n\rightarrow \infty}\frac{\int_{\mathbb{R}^N}|\nabla u_n^2|^2dx}{\int_{\mathbb{R}^N}|u_n|^{4+\frac{4}{N}}dx}=\frac{Na_{q^*}}{N+1}.$$
Setting \begin{equation}\label{eq2.15}w_n=\varepsilon_n^Nu_n^2(\varepsilon_n x) \text{ with  }\varepsilon_n=\left(\int_{\mathbb{R}^N}|\nabla u_n^2|^2dx\right)^{-\frac{1}{N+2}}.\end{equation}
We have
$$\int_{\mathbb{R}^N}|w_n|dx=\int_{\mathbb{R}^N}|u_n|^{2}dx=1,$$
$$\int_{\mathbb{R}^N}|\nabla w_n|^2dx=\varepsilon_n^{N+2}\int_{\mathbb{R}^N}|\nabla u_n^2|^2dx=1,$$
and
$$\int_{\mathbb{R}^N}|w_n|^{2+\frac{2}{N}}dx=\varepsilon_n^{N+2}\int_{\mathbb{R}^N}|u_n|^{4+\frac{4}{N}}dx\overset{n}\rightarrow \frac{N+1}{Na_{q^*}}.$$
Thus,
$$E_{q^*}^{a_{q^*}}(u_n)=\frac{1}{4}\varepsilon_n^{-(N+2)}F(w_n)+\frac{1}{2}\int_{\mathbb{R}^N}(\nabla u_n|^2+V(x)u_n^2)dx=d_{a_{q^*}}(q^*)+o(1).$$
Consequently,
$$F(w_n)=4\varepsilon_n^{N+2}\left[d_{a_{q^*}}(q^*)-\frac{1}{2}\int_{\mathbb{R}^N}(\nabla u_n|^2+V(x)u_n^2)dx+o(1)\right]\overset{n}\rightarrow 0=m(1).$$
This implies that  $\{w_n\}$ is a  minimizing sequence of $m(1)$ and $\int_{\mathbb{R}^N}|\nabla w_n|^2dx=1$.
It then follows from  Lemma \ref{l2.5} that there exists $\{y_n\}\subset\R^N$ such that   $w_n(\cdot+y_n)\overset{n}\rightarrow w_0=\frac{\lambda_0^N}{|v_{q^*}|_{L^1}}v_{q^*}(\lambda_0 x)$ in $\mathcal{D}^{2,1}(\R^N)$. However, it follows from (\ref{eq2.15}) that
$$\liminf_{n\rightarrow \infty}\varepsilon_n^{2}\int_{\mathbb{R}^N}|\nabla u_n|^2dx=\liminf_{n\rightarrow\infty} \int_{\mathbb{R}^N}|\nabla w_n^{\frac{1}{2}}|^2dx\geq \int_{\mathbb{R}^N}|\nabla w_0^{\frac{1}{2}}|^2dx\geq C>0.$$
This  indicates that $$\int_{\mathbb{R}^N}|\nabla u_n|^2dx\geq C\varepsilon_n^{-2}\overset{n}\rightarrow +\infty,$$
 which contradicts the fact that $\{u_n\}$ is bounded in $H$. Thus, claim (\ref{eq2.14}) is proved.  Furthermore, one can similar to the argument of Lemma \ref{exi1} (i) obtain that  and $u_n\overset{n}\rightarrow u$ in $X$ with $u$ being  a minimizer of $d_{a_{q^*}}(q^*)$.
\end{proof}

\noindent{\bf Proof of Theorem \ref{t1.1}:} Lemma \ref{exi1} together with Lemma \ref{ex3} indicates the conclusions of Theorem \ref{t1.1}.\qed

\section{Case of $a<a_{q^*}$: Proof of Theorem \ref{th1.2}. }\label{se3}
 The aim of this section is to prove that when $a<a_{q^*}$ is fixed,  all minimizers of (\ref{e1.1}) are compact in the space $X$ as $q\nearrow q^*$, which gives the proof of Theorem  \ref{th1.2}.\\

\noindent\textbf{Proof of Theorem \ref{th1.2}:} For any $\eta(x)\in C_0^\infty(\mathbb{R}^N)$ and $|\eta(x)|_{L^2}^2=1$. We can find a constant $C>0$ independent of $q$, such that
\begin{equation}\label{2.10}
d_a(q)\leq E_q^a(\eta)\leq C<\infty.
\end{equation}
Assume that $u_q$ is a nonnegative minimizer of (\ref{e1.1}), we deduce from (\ref{1.5}) that
\begin{align}
&\frac{1}{2}\int_{\mathbb{R}^N} |\nabla
u_q|^2dx+\frac{1}{4}\int_{\mathbb{R}^N} |\nabla
u_q^2|^2dx+\frac{1}{2}\int_{\mathbb{R}^N}
V(x)|u_q|^2dx \nonumber\\
&=d_a(q)+\frac{a}{q+2}\int_{\mathbb{R}^N} |u_q|^{q+2}dx\leq C+\frac{a}{q+2} \frac{1}{\lambda_qa_{q}}|\nabla
u_q^2|_{L^2}^{\frac{2q}{q^*}}.\label{eq3.2}
 \end{align}
This implies that
\begin{equation*}
\frac{1}{4}\int_{\mathbb{R}^N} |\nabla
u_q^2|^2dx\leq C+\frac{a}{q+2} \frac{1}{\lambda_qa_{q}}|\nabla
u_q^2|_{L^2}^{\frac{2q}{q^*}}.
\end{equation*}
We claim that
\begin{equation}\label{2.11}
\limsup_{q\nearrow q^*}\int_{\mathbb{R}^N} |\nabla
u_q^2|^2dx\le C<\infty.
\end{equation}
For otherwise, if $$\int_{\mathbb{R}^N} |\nabla
u_q^2|^2dx\triangleq M_q\rightarrow\infty \text{ as }q\nearrow q^*.$$ On  one hand, we know from (\ref{eq3.2}) that
\begin{equation*}
\begin{split}
M_q\leq C+\frac{a}{q+2} \frac{4}{\lambda_qa_{q}}|\nabla
u_q^2|_{L^2}^{\frac{2q}{q^*}}\leq \frac{1}{2}\left(1-\frac{a}{a_{q^*}}\right)M_q^{\frac{q}{q^*}}+\frac{a}{q+2} \frac{4}{\lambda_qa_{q}}M_q^{\frac{q}{q^*}},
\end{split}
\end{equation*}
i.e.,
\begin{equation}\label{2.12}
M_q\leq \left[\frac{1}{2}\left(1-\frac{a}{a_{q^*}}\right)+\frac{a}{q+2} \frac{4}{\lambda_qa_{q}}\right]^{\frac{q^*}{q^*-q}}.
\end{equation}
On the other hand, from the definitions of $\lambda_q$ and $a_q$ in (\ref{eq1.12}) we get that
\begin{equation} \label{eq3.4}\frac{1}{q+2}\frac{1}{\lambda_q}\rightarrow \frac{1}{4}\ \text{ and } \  a_q\rightarrow a_{q^*} \ \text{ as } q\nearrow q^*.\end{equation}
Thus,
$$\lim_{q\nearrow q^*}\left[\frac{1}{2}\left(1-\frac{a}{a_{q^*}}\right)+\frac{a}{q+2} \frac{4}{\lambda_qa_{q}}\right]=\left[\frac{1}{2}\left(1-\frac{a}{a_{q^*}}\right)+\frac{a}{a_{q^*}}\right]<1.$$
This together with (\ref{2.12})  implies that
\begin{equation*}
M_q\leq \left[\frac{1}{2}\left(1-\frac{a}{a_{q^*}}\right)+\frac{a}{q+2} \frac{4}{\lambda_qa_{q}}\right]^{\frac{q^*}{q^*-q}}\rightarrow 0 \ \text{ as }q\nearrow q^*,
\end{equation*}
this is impossible. Hence, (\ref{2.11}) is obtained and   it is easy to further show that
\begin{equation*}
\frac{1}{2}\int_{\mathbb{R}^N} (|\nabla
u_q|^2dx+
V(x)|u_q|^2)dx\le C<\infty.
\end{equation*}
Which means that $\{u_q\}$ is bounded in $H^1(\mathbb{R}^N)$. As a consequence,  there exists a subsequence (denoted still by $\{u_q\}$ ), and $0\le  u_0\in H$, such that
\begin{equation*}
u_q\rightharpoonup u_0\ in\ H; \ u_q\rightarrow u_0\ in\ L^p(\mathbb{R}^N),\ \forall\ 2\leq p<2^*.
\end{equation*}
By applying  Lebesgue's dominated convergence theorem, one can obtain from above that
\begin{equation*}
\lim_{q\nearrow q^*}\int_{\mathbb{R}^N} |u_q|^{q+2}dx=\int_{\mathbb{R}^N} |u_0|^{q^*+2}dx.
\end{equation*}
%
Therefore,
\begin{align}
&\lim_{q\nearrow q^*}d_a(q)
=\lim_{q\nearrow q^*}\Big[\frac{1}{2}\int_{\mathbb{R}^N} \left(|\nabla
u_q|^2+V(x)|u_q|^2\right)dx+\frac{1}{4}\int_{\mathbb{R}^N} |\nabla
u_q^2|^2dx-\frac{a}{q+2}\int_{\mathbb{R}^N} |u_q|^{q+2}dx\Big]\nonumber\\
 &\geq \frac{1}{2}\int_{\mathbb{R}^N} \left(|\nabla
u_0|^2+V(x)|u_0|^2\right)dx+\frac{1}{4}\int_{\mathbb{R}^N} |\nabla
u_0^2|^2dx-\frac{a}{q^*+2}\int_{\mathbb{R}^N} |u_0|^{q^*+2}dx\nonumber\\
&=E_{q^*}(u_0)\geq d_a(q^*).\label{2.13}
 \end{align}
On the other hand, for any $\varepsilon>0$, there exists $u_\varepsilon\in X$ and $|u_\varepsilon|_{L^2}^2=1$, such that
\begin{equation*}
E_{q^*}^a (u_\varepsilon)\leq d_a(q^*)+\varepsilon.
\end{equation*}
Then,
\begin{equation*}
\begin{split}
&\lim_{q\nearrow q^*}d_a(q)\leq \lim_{q\nearrow q^*}E_q(u_\varepsilon)\\
&\leq\lim_{q\nearrow q^*}\left\{\frac{1}{2}\int_{\mathbb{R}^N} \left(|\nabla
u_\varepsilon|^2+V(x)|u_\varepsilon|^2\right)dx+\frac{1}{4}\int_{\mathbb{R}^N} |\nabla
u_\varepsilon^2|^2dx-\frac{a}{q+2}\int_{\mathbb{R}^N} |u_\varepsilon|^{q+2}dx\right\}\\
&=E_{q^*}^a(u_\varepsilon)+\lim_{q\nearrow q^*}\left\{\frac{a}{q^*+2}\int_{\mathbb{R}^N} |u_\varepsilon|^{q^*+2}dx-\frac{a}{q+2}\int_{\mathbb{R}^N} |u_\varepsilon|^{q+2}dx\right\}\\
&\leq d_a(q^*)+\varepsilon,
 \end{split}
 \end{equation*}
Letting $\varepsilon\rightarrow 0$, this inequality together with (\ref{2.13}) gives that
\begin{equation*}
\lim_{q\rightarrow q^*}d_a(q)=d_a(q^*)=E_{q^*}^a(u_0).
\end{equation*}
Hence, $u_0$ is a nonnegative minimizer of $d_a(q^*)$ and $u_q\rightarrow u_0$ in $X$ as $q\nearrow q^*$.
\qed


\section{Case of $a>a_{q^*}$: Proof of Theorem \ref{th1.3}.}\label{se4}
This section is devoted to proving Theorem \ref{th1.3} on the blow-up  of minimizers for (\ref{e1.1}) as $q\nearrow q^*$ for the case of $a>a_{q^*}$. For this purpose, we introduce the following functional
\begin{equation*}
\tilde{E}_q^a(u)=\frac{1}{2}\int_{\mathbb{R}^N} |\nabla
u|^2dx+\frac{1}{4}\int_{\mathbb{R}^N} |\nabla
u^2|^2dx-\frac{a}{q+2}\int_{\mathbb{R}^N} |u|^{q+2}dx£¬
\end{equation*}
and consider  the following minimization problem
\begin{equation}\label{eq4.1}
\tilde{d}_a(q)=\inf_{u\in M}\tilde{E}_q^a(u).
\end{equation}
We remark that when $a>a_{q^*}$, it  follows from (\ref{eq3.4}) that
\begin{equation}\label{eq4.22}
 \lim_{q\nearrow q^*}\frac{4aq}{q^*\lambda_q a_q(q+2)}=\frac{a}{a_{q^*}}>1.
 \end{equation}
 We then  obtain from Lemma \ref{le4.1} below that $\tilde{d}_a(q)<0$ if $q<q^*$ and is closed to $q^*$. As a consequence, one can use  \cite[Lemma 1.1]{jlw}
to deduce that (\ref{eq4.1}) has at least one minimizer.

\subsection{The blow-up analysis for the  minimizers of (\ref{eq4.1}).}\label{S4.1}
In this subsection, we study the following concentration phenomena for the minimizers of (\ref{eq4.1}) as $q\nearrow q^*$, which is crucial for the proving of Theorem \ref{th1.3}.
\begin{theorem}\label{th4.1}
Let $a>a_{q^*}$  and  $u_q$ be a nonnegative minimizer of (\ref{eq4.1}) with $q<q^*$. For any sequence of $\{u_q\}$,  there exists a subsequence, denoted still by $\{u_q\}$,  and  $\{y_{\varepsilon_q}\}\subset \R^N$ such that the scaling
\begin{equation}\label{eq4.30}
w_q=\varepsilon_q^{\frac{N}{2}}u_q(\varepsilon_q
x+\varepsilon_qy_{\varepsilon_q})
\end{equation}
satisfies
\begin{equation}\label{eq4.3}
\lim_{q\nearrow q^*} \varepsilon_q^{N}u_q^2(\varepsilon_q
x+\varepsilon_qy_{\varepsilon_q})=w_0^2:=\frac{\lambda^N}{|v_{q^*}|_{L^1}}v_{q^*}\Big(\lambda |x-x_0|\Big) \text{ in } \mathcal {D}^{2,1}(\mathbb{R}^N),
\end{equation}
where $\lambda,\ \varepsilon_q$ is given by (\ref{eq1.15}) and  $x_0\in\R^N$. Moreover, there exist positive constants  $C,\mu $ and $R$ independent of $q$, such that
\begin{equation}\label{eq4.4}
w_q(x)\leq C e^{-\mu |x|} \text{ for any } |x|>R \text{ as } q\nearrow q^*.
\end{equation}
\end{theorem}

To prove the above theorem, we first  give the  following energy estimate of $\tilde{d}_a(q)$.
\begin{lemma}\label{le4.1}
Let $a>a_{q^*}$ be fixed. Then,
$$\tilde{d}_a(q)=-\frac{q^*-q}{4q}\left(\frac{4aq}{q^*a_q\lambda_q(q+2)}\right)^{\frac{q^*}{q^*-q}}(1+o(1))(\to-\infty) \ \text{ as } q\nearrow q^*.$$
\end{lemma}
\begin{proof}
For  any $ u\in M$, we obtain from (\ref{1.5}) that

\begin{align}\label{3.1}
\tilde{E}_q^a(u)&=\frac{1}{2}\int_{\mathbb{R}^N} |\nabla
u|^2dx+\frac{1}{4}\int_{\mathbb{R}^N} |\nabla
u^2|^2dx-\frac{a}{q+2}\int_{\mathbb{R}^N} |u|^{q+2}d\nonumber\\
&\geq \frac{1}{4}\int_{\mathbb{R}^N} |\nabla
u^2|^2dx-\frac{a}{(q+2)\lambda_q a_q}\left(\int_{\mathbb{R}^N} |\nabla
u^2|^2dx\right)^{\frac{q}{q^*}}
\end{align}
Setting $$\int_{\mathbb{R}^N} |\nabla
u^2|^2dx=t \text{ and }\ g(t)=\frac{1}{4}t-\frac{a}{(q+2)\lambda_q a_q}t^{\frac{q}{q^*}},\ t\in(0,+\infty).$$ It is easy to know that  $g(t)$ gets its minimum at a unique  point \begin{equation}\label{eq4.2} t_q=\left(\frac{4aq}{q^*\lambda_q a_q(q+2)}\right)^{\frac{q^*}{q*-q}}, \end{equation} Hence,
\begin{equation}\label{3.2}
g(t)\geq g(t_q)=-\frac{q^*-q}{4q}\left(\frac{4aq}{q^*\lambda_q a_q(q+2)}\right)^{\frac{q^*}{q*-q}}.
\end{equation}
From (\ref{3.1}) and (\ref{3.2}), we know that
$$\tilde{d}_a(q)\geq -\frac{q^*-q}{4q}\left(\frac{4aq}{q^*\lambda_q a_q(q+2)}\right)^{\frac{q^*}{q*-q}}.$$
This gives the lower bound of $\tilde{d}_a(q)$. We next will prove the upper bound.

Let $$u_\tau=\frac{\tau^{\frac{N}{2}}}{|v_{q}|_{L^1}^{\frac{1}{2}}}\sqrt{v_{q}(\tau x)},\ \tau>0,$$
where $v_q(x)$ is the unique nonnegative solution of (\ref{1.3}). Then,
 $\int_{\mathbb{R}^N}u_\tau^2dx=1$ and it follows from (\ref{1.4}) that
\begin{equation*}
\int_{\mathbb{R}^N}|\nabla
u_\tau^2|^2dx=\frac{\tau^{N+2}}{|v_{q}|_{L^1}^2}\int_{\mathbb{R}^N}|\nabla
v_{q}|^2dx=\frac{\theta_q\tau^{N+2}}{(1-\theta_q)|v_{q}|_{L^1}},
\end{equation*}
\begin{equation*}
\int_{\mathbb{R}^N}u_\tau^{q+2}dx=\frac{\tau^{\frac{Nq}{2}}}{|v_{q}|_{L^1}^{\frac{q+2}{2}}}\int_{\mathbb{R}^N}v_{q}^{\frac{q+2}{2}}dx
=\frac{\tau^{\frac{Nq}{2}}}{(1-\theta_q)|v_{q}|_{L^1}^{\frac{q}{2}}},
\end{equation*}
\begin{equation*}
\int_{\mathbb{R}^N}|\nabla
u_\tau|^2dx=\frac{\tau^{2}}{|v_{q}|_{L^1}}\int_{\mathbb{R}^N}|\nabla
\sqrt{v_{q}}|^2dx\leq c_1\tau^2.
\end{equation*}
Taking $\tau=\left(\frac{(1-\theta_q)t_q |v_q|_{L^1}}{\theta_q}\right)^{\frac{1}{N+2}}$ with $t_q$  given by (\ref{eq4.2}), we then have
\begin{align*}
&\frac{1}{4}\int_{\mathbb{R}^N}|\nabla
u_\tau^2|^2dx-\frac{a}{q+2}\int_{\mathbb{R}^N}u_\tau^{q+2}dx=\frac{\theta_q\tau^{N+2}}{4(1-\theta_q)|v_{q}|_{L^1}}-\frac{a}{q+2}\frac{\tau^{\frac{Nq}{2}}}{(1-\theta_q)|v_{q}|_{L^1}^{\frac{q}{2}}}\nonumber\\
&=-\frac{q^*-q}{4q}\left(\frac{4aq}{q^*\lambda_q a_q(q+2)}\right)^{\frac{q^*}{q*-q}},
\end{align*}
 and
 \begin{equation*}
\int_{\mathbb{R}^N}|\nabla u_\tau|^2dx\leq c_0\tau^2\leq c_1\left(\frac{4aq}{q^*\lambda_q a_q(q+2)}\right)^{\frac{q^*}{q*-q}\cdot \frac{2}{N+2}}.
\end{equation*}
 Therefore,
 \begin{equation}\label{3.9}
 \tilde{E}_q^a(u_\tau)\leq -\frac{q^*-q}{4q}\left(\frac{4aq}{q^*\lambda_q a_q(q+2)}\right)^{\frac{q^*}{q*-q}}+c_1\left(\frac{4aq}{q^*\lambda_q a_q(q+2)}\right)^{\frac{q^*}{q*-q}\cdot \frac{2}{N+2}}.
 \end{equation}
From (\ref{eq4.22}) we see that
 $\left(\frac{4aq}{q^*\lambda_q a_q(q+2)}\right)^{\frac{q^*}{q^*-q}}\rightarrow +\infty$ as $q\nearrow q^*$. 
This together with (\ref{3.9}) implies  that
 \begin{equation*}
 \tilde{d}_a(q)\leq\tilde{E}_q^a(u_\tau)\leq -\frac{q^*-q}{4q}\left(\frac{4aq}{q^*\lambda_q a_q(q+2)}\right)^{\frac{q^*}{q*-q}}(1+o(1)).
 \end{equation*}
 Which gives the upper bound of $ \tilde{d}_a(q)$ and the lemma is
 proved.
\end{proof}

\begin{lemma}\label{le4.2}
Let $a>a_{q^*}$ be fixed and $u_q$ be a nonnegative minimizer of $\tilde{d}_a(q)$. Then
\begin{equation}\label{3.10}
\int_{\mathbb{R}^N}|\nabla u_q^2|^2dx\approx \frac{4a}{q+2}\int_{\mathbb{R}^N}u_q^{q+2}dx\approx \left(\frac{4aq}{q^*\lambda_q a_q(q+2)}\right)^{\frac{q^*}{q*-q}}:=t_q
\end{equation}
and
\begin{equation}\label{3.11}
\frac{\int_{\mathbb{R}^N}|\nabla u_q|^2dx}{\int_{\mathbb{R}^N}|\nabla u_q^2|^2dx}\rightarrow 0 \ \text{ as }\ q\nearrow q^*.
\end{equation}
 Here $a\approx b$ means that $\frac{a}{b}\rightarrow 1$ as $q\nearrow q^*$.
\end{lemma}
\begin{proof}
From (\ref{3.1}), we have
\begin{equation}\label{3.12}
\begin{split}
\tilde d_a(q)=\tilde{E}_q^a(u_q)&\geq \frac{1}{4}\int_{\mathbb{R}^N} |\nabla
u_q^2|^2dx-\frac{a}{(q+2)\lambda_q a_q}\left(\int_{\mathbb{R}^N} |\nabla
u_q^2|^2dx\right)^{\frac{q}{q^*}}\\
&:=g(t)=\frac{1}{4}t-\frac{a}{(q+2)\lambda_q a_q}t^{\frac{q}{q^*}} \text{ with } t=\int_{\mathbb{R}^N} |\nabla
u_q^2|^2dx.
\end{split}
\end{equation}
We first prove that
\begin{equation}\label{eq4.7}
\int_{\mathbb{R}^N}|\nabla u_q^2|^2dx\approx \left(\frac{4aq}{q^*\lambda_q a_q(q+2)}\right)^{\frac{q^*}{q*-q}}=t_q \ \text{ as }\ \nearrow q^*.
\end{equation}
For otherwise, if it is  false, then in subsequence sense there holds that
$$\lim_{q\nearrow q^*}\frac{\int_{\mathbb{R}^N}|\nabla u_q^2|^2dx}{t_q}=\gamma \in[0,1)\cup (1, +\infty).$$
If $\gamma\in[0,1)$, then,
\begin{align}
\lim_{q\nearrow q^*}\frac{g(\gamma t_q)}{g(t_q)}&=\lim_{q\nearrow q^*}\frac{\gamma t_q-\frac{4a}{(q+2)\lambda_q a_q}(\gamma t_q)^{\frac{q}{q^*}}}{t_q-\frac{4a}{(q+2)\lambda_q a_q}t_q^{\frac{q}{q^*}}}=\lim_{q\nearrow q^*}\frac{q^*\gamma^{\frac{q}{q^*}}-q\gamma}{q^*-q}\nonumber\\
&=\gamma(-\ln \gamma+1)\in [0,1).  \label{eq4.8}
\end{align}
Let $\delta:=\gamma(-\ln \gamma+1)\in [0,1)$, we then obtain from (\ref{3.12}) and(\ref{eq4.8}) that
$$ \tilde{d}_a(q)\geq \big(1+o(1)\big)g(\gamma t_q)\geq \frac{1+\delta}{2} g(t_q)=-\frac{1+\delta}{2}\cdot\frac{q^*-q}{4q}\left(\frac{4aq}{q^*\lambda_q a_q(q+2)}\right)^{\frac{q^*}{q*-q}}.$$
This contradicts Lemma 3.1. Similar to the above argument, one can prove also that $\gamma \in(1,+\infty)$ cannot occur. Thus, (\ref{eq4.7}) is proved.


We next try to prove (\ref{3.11}).  On the contrary, if it is false, then  $\exists\ \beta>0$, such that   $\int_{\mathbb{R}^N}|\nabla u_q|^2dx\geq \frac{\beta}{2}\int_{\mathbb{R}^N}|\nabla u_q^2|^2dx$. Therefore,
\begin{equation}\label{eq4.155}
\begin{split}
0\geq\tilde{d}_a(q)&=\frac{1}{2}\int_{\mathbb{R}^N} |\nabla
u_q|^2dx+\frac{1}{4}\int_{\mathbb{R}^N} |\nabla
u_q^2|^2dx-\frac{a}{q+2}\int_{\mathbb{R}^N} |u_q|^{q+2}dx\\
&\geq \frac{1+\beta}{4}\int_{\mathbb{R}^N} |\nabla
u_q^2|^2dx-\frac{a}{(q+2)\lambda_q a_q}\left(\int_{\mathbb{R}^N} |\nabla
u_q^2|^2dx\right)^{\frac{q}{q^*}}\\
&\triangleq \tilde{g}(s),
\end{split}
\end{equation}
by setting $\tilde{g}(s)=\frac{1+\beta}{4}s-\frac{a}{(q+2)\lambda_q
a_q}s^{\frac{q}{q^*}}$ with $s=\int_{\mathbb{R}^N} |\nabla
u_q^2|^2dx$. One can easily check that $\tilde{g}(s)\geq
\tilde{g}(s_q)$ with $s_q=\left(\frac{4aq}{q^*\lambda_q
a_q(q+2)(1+\beta)}\right)^{\frac{q^*}{q*-q}}$ and
$$\tilde{g}(s_q)=-\frac{(1+\beta)(q^*-q)}{4q}\left(\frac{4aq}{q^*\lambda_q a_q(q+2)(1+\beta)}\right)^{\frac{q^*}{q*-q}}\triangleq A.$$
However
$$\frac{A}{-\frac{(q^*-q)}{4q}\left(\frac{4aq}{q^*\lambda_q a_q(q+2)}\right)^{\frac{q^*}{q*-q}}}=(1+\beta)\left(\frac{1}{1+\beta}\right)^{\frac{q^*}{q*-q}}\rightarrow 0 \ \text{ as }q\nearrow q^*.$$
This together with (\ref{eq4.155}) indicates that
$$\frac{\tilde{d}_a(q)}{-\frac{(q^*-q)}{4q}\left(\frac{4aq}{q^*\lambda_q a_q(q+2)}\right)^{\frac{q^*}{q*-q}}}\rightarrow 0\ \text{ as }q\nearrow q^*,$$
 which  contradicts Lemma 3.1, and (\ref{3.11}) is proved.

Finally, taking
$$\frac{\tilde{d}_a(q)}{t_q}=\frac{1}{2t_q}\int_{\mathbb{R}^N} |\nabla
u_q|^2dx+\frac{1}{4t_q}\int_{\mathbb{R}^N} |\nabla
u_q^2|^2dx-\frac{a}{(q+2)t_q}\int_{\mathbb{R}^N} |u_q|^{q+2}dx.$$
we then obtain from (\ref{3.11}) and (\ref{eq4.7}) that
$$\frac{4a}{(q+2)t_q}\int_{\mathbb{R}^N} |u_q|^{q+2}dx\rightarrow 1\ \text{ as }q\nearrow q^*.$$
This gives  (\ref{3.10}).  The proof of this lemma is finished.
\end{proof}

Applying Lemmas \ref{le4.1} and \ref{le4.2}, we end this subsection by proving Theorem \ref{th4.1}.\\

\noindent\textbf{Proof of Theorem \ref{th4.1}: }
Set
$$\varepsilon_q=t_q^{-\frac{1}{N+2}}=t_q^{-\frac{2}{Nq^*}} \text{ with } t_q \text{ given by  (\ref{eq4.2})},$$
 it follows from (\ref{eq4.22}) that $\lim_{q\nearrow q^*}\varepsilon_q=0$. Let $u_q$ be a nonnegative minimizer of $\tilde{d}_a(q)$ and define
$$\tilde{w}_q=\varepsilon_q^{\frac{N}{2}}u_q(\varepsilon_qx).$$

From Lemma \ref{le4.2}, we have
\begin{equation}\label{eq4.15}
\int_{\mathbb{R}^N} |\nabla
\tilde{w}_q^2|^2dx=\varepsilon_q^{N+2}\int_{\mathbb{R}^N} |\nabla
u_q^2|^2dx\approx t_q^{-1}t_q=1,
\end{equation}
\begin{equation}\label{3.13}
\int_{\mathbb{R}^N}
|\tilde{w}_q|^{q+2}dx=\varepsilon_q^{\frac{qN}{2}}\int_{\mathbb{R}^N}
|u_q|^{q+2}dx\approx\frac{q^*+2}{4a_{q^*}},
\end{equation}
and
\begin{equation}\label{eq4.16}
\int_{\mathbb{R}^N} |\nabla
\tilde{w}_q|^2dx=\varepsilon_q^2\int_{\mathbb{R}^N} |\nabla
u_q|^2dx=o(1)\varepsilon_q^{-N}.
\end{equation}
Since $u_q$ is a minimizer of $\tilde{d}_a(q)$, then there exists  $\mu_q\in \mathbb{R}$, such that
\begin{equation}\label{3.14}
-\triangle u_q-\triangle (u_q^2)u_q=\mu_qu_q +a|u_q|^{q}u_q.
\end{equation}
Therefore,
\begin{align*}
\mu_q&=\int_{\mathbb{R}^N} |\nabla
u_q|^2dx+\int_{\mathbb{R}^N} |\nabla
u_q^2|^2dx-a\int_{\mathbb{R}^N} |u_q|^{q+2}dx\\
&=4\tilde{d}_a(q)-\int_{\mathbb{R}^N} |\nabla
u_q|^2dx+\frac{a(2-q)}{q+2}\int_{\mathbb{R}^N} |u_q|^{q+2}dx.
\end{align*}
From Lemmas \ref{le4.1} and \ref{le4.2}, we get that
\begin{align}\label{3.15}
\mu_q\varepsilon_q^{N+2}=\mu_qt_q^{-1}=-\frac{q^*-q}{4q}+o(1)+\frac{2-q}{4}(1+o(1))\rightarrow-\frac{1}{N} \ \text{ as }\ q\nearrow q^*.
\end{align}

Using (\ref{3.13}) and \cite[Lemma I.1]{lions1}, we see that there exists $\{y_{\varepsilon_q}\}\subset \mathbb{R}^N$ and $R, \eta>0$, s.t.
\begin{equation*}
\liminf_{q\nearrow q^*}\int_{B_R(y_{\varepsilon_q})}
|\tilde{w}_q|^{2}dx>\eta>0.
\end{equation*}
Let \begin{equation}\label{eq4.18}w_q=\tilde{w}_q(x+y_{\varepsilon_q})=\varepsilon_q^{\frac{N}{2}}u_q(\varepsilon_q
x+\varepsilon_qy_{\varepsilon_q}),\end{equation}
 then,
\begin{equation}\label{3.16}
\liminf_{q\nearrow q^*}\int_{B_R(0)} |w_q|^{2}dx>\eta>0.
\end{equation}
From (\ref{3.14}), we see that $w_q(x)$ satisfies
\begin{equation}\label{3.17}
-\varepsilon_q^N\triangle w_q-\triangle
(w_q^2)w_q=\mu_q\varepsilon_q^{N+2}w_q
+a\varepsilon_q^{N+2-\frac{Nq}{2}}w_q^{q+1}.
\end{equation}
Note that $N+2=\frac{Nq^*}{2}$, we can deduce that
$$
\varepsilon_q^{N+2-\frac{Nq}{2}}=\varepsilon_q^{\frac{N(q^*q)}{2}}=t_q^{-\frac{q^*-q}{q^*}}\\
=\left(\frac{4a
q}{q^*\lambda_qa_q(q+2)}\right)^{-1}=\frac{q^*\lambda_qa_q(q+2)}{4a
q}.
$$
Consequently,
\begin{equation}\label{3.18}
\lim_{q\nearrow q^*}a\varepsilon_q^{N+2-\frac{Nq}{2}}=\lim_{q\nearrow q^*}\frac{q^*\lambda_qa_q}{q}= a_{q^*}.
\end{equation}
 Moreover, for any $\varphi \in C_c^\infty(\mathbb{R}^N)$, we deduce from (\ref{eq4.16}) and (\ref{eq4.18}) that
\begin{align*}
\left|\varepsilon_q^N\int_{\mathbb{R}^N}\nabla w_q\nabla \varphi dx\right|&\leq C\varepsilon_q^N\left(\int_{\mathbb{R}^N}|\nabla w_q|^2dx\right)^{\frac{1}{2}}= o(1)\varepsilon_q^{\frac{N}{2}}\to 0 \ \text{ as } q\nearrow q^*.
\end{align*}
 By passing to subsequence, it then follows from (\ref{3.15})-(\ref{3.18}) that
 $$w_q^2\rightharpoonup w_0^2 \ \text{ in }\ \mathcal {D}^{2,1} (\R^N)\ \text{ as }\ q\nearrow q^*,$$ where $0\leq  w_0\not\equiv0$  satisfies
\begin{equation*}
-\triangle (w_0^2)w_0=-\frac{1}{N}w_0+a_{q^*}w_0^{3+\frac{4}{N}},
\end{equation*}
i.e.
\begin{equation}\label{3.19}
-\triangle
(w_0^2)=-\frac{1}{N}+a_{q^*}\left(w_0^2\right)^{1+\frac{2}{N}}.
\end{equation}
Using classical  Pohozaev identities, we obtain that
\begin{equation*}
         \int_{\mathbb{R}^N}|\nabla w_0^2| ^2dx= \int_{\mathbb{R}^N}w_0^2dx \text{ and } \int_{\mathbb{R}^N}\left(w_0^2\right)^{\frac{q^*+2}{2}}dx=\frac{N+1}{Na_{q^*}}\int_{\mathbb{R}^N}w_0^2dx.
                          \end{equation*}
%
Recalling the  Gagliardo-Nirenberg inequality (\ref{2.2}), we then have
  \begin{equation}\label{eq4.25}
  \frac{Na_{q^*}}{N+1}\leq \frac{\int_{\mathbb{R}^N}|\nabla w_0^2| ^2dx\left(\int_{\mathbb{R}^N}w_0^2dx\right)^{\frac{2}{N}}}{\int_{\mathbb{R}^N}\left(w_0^2\right)^{2+\frac{2}{N}}}= \frac{Na_{q^*}}{N+1}\left(\int_{\mathbb{R}^N}w_0^2dx\right)^{\frac{2}{N}}.
  \end{equation}
This indicates that
$$\int_{\mathbb{R}^N}w_0^2dx\geq 1.$$
On the other hand, there always holds that
$$\int_{\mathbb{R}^N}w_0^2dx\leq \liminf_{q\nearrow q^*}\int_{\mathbb{R}^N}w_q^2dx=1.$$
Consequently, we have
\begin{equation}\label{eq4.26}\int_{\mathbb{R}^N}w_0^2dx= 1,\end{equation}
and thus
$$w_q\rightarrow w_0\ \text{ in } \ L^2(\R^N) \  \text{ as } \ q\nearrow q^*.$$
It then follows  from (\ref{eq4.15}), (\ref{3.17}) and (\ref{3.19}) that
\begin{equation}\label{eq4.27}\liminf_{q\nearrow q^*}\int_{\mathbb{R}^N}|\nabla w_q^2| ^2dx=\int_{\mathbb{R}^N}|\nabla w_0^2| ^2dx=1.\end{equation}
This means that $$w_q^2\rightarrow w_0^2 \text{ in }\mathcal {D}^{2,1}(\mathbb{R}^N).$$
Moreover, it follows from (\ref{eq4.25}) and (\ref{eq4.26}) that $w_0^2\geq0$ is an optimizer of (\ref{2.2}), thus it must be of the form
$$w_0^2=\frac{\lambda^N}{|v_{q^*}|_{L^1}}v_{q^*}\Big(\lambda |x-x_0|\Big),$$
where $\lambda=\left(\frac{|v_{q^*}|_{L^1}}{N}\right)^\frac{1}{N+2}$ follows from (\ref{eq4.27}).  This completes the proof of (\ref{eq4.3}).

Now, it remains to prove (\ref{eq4.4}) to complete the proof of Theorem \ref{th4.1}. Indeed, from (\ref{3.17}) and (\ref{3.18}) we see that
\begin{equation*}
-\triangle
(w_q^2)\leq c(x)w_q^2 \ \text{ with } c(x)=2a_{q^*}w_q^{q-2}.
\end{equation*}
Similar to the proof of \cite[Theorem 1.1]{gzz}, one can use
 DeGiorgi-Nash-Moser theory  as well as the comparison principle to deduce that there exists $C,\beta,R>0$  independent of $q$, such that
\begin{equation*}
w_q^2(x)\leq C e^{-\beta |x|} \text{ for any } |x|>R \ \text{ as }\ q\nearrow q^*.
\end{equation*}
This gives (\ref{eq4.4}) by taking $\mu=\frac{\beta}{2}$.
\qed

\subsection{Proof of Theorem \ref{th1.3}.}\label{S4.2}
This subsection is devoted to proving Theorem \ref{th1.3} on the blow-up behavior of minimizers for (\ref{e1.1}) as $q\nearrow q^*$. We first give  precise  energy estimates of $d_{a}(q)$ in the following lemma.
%
\begin{lemma}\label{l4.1}
Let $a>a_{q^*}$ be fixed and $\bar u_q(x)$ be a nonnegative minimizer of $d_a(q)$. Then,
\begin{equation}\label{4.2}
0\leq d_a(q)-\tilde{d}_a(q)\rightarrow 0\ \text{ as }\ q\nearrow q^*,
\end{equation}
and
\begin{equation}\label{4.3}
\int_{\mathbb{R}^N}V(x)\bar{u}_q^2dx\rightarrow 0\ \text{ as }\ q\nearrow q^*.
\end{equation}
\end{lemma}

\begin{proof}
Let $\varphi(x)$ be a cut-off function such that $\varphi(x)\equiv1$ if $|x|<1$ and $\varphi(x)\equiv0$ if $|x|>1$. As in Subsection \ref{S4.1}, we still denote $u_q$ to be a nonnegative minimizer of $\tilde d_a(q)$ and let $w_q$ be  given by (\ref{eq4.30}).  For any $x_0\in\R^N$, we set
\begin{equation*}
\tilde  u_q(x)=A_q\varphi(x-x_0)\varepsilon_q^{-\frac{N}{2}}w_q\big(\frac{x-x_0}{\varepsilon_q}\big)=A_q\varphi(x-x_0)u_q(x-x_0+\varepsilon_q y_{\varepsilon_q}),
\end{equation*}
where $A_q\geq1$ such that $\int_{\R^N}\tilde u_q^2dx\equiv1$. Using the exponential decay of $w_q$ in (\ref{eq4.4}), we have

\begin{equation}
0\leq A_q^2-1=\frac{\int_{|x|\geq1}\varphi(\varepsilon_q x)w_q^2(x)dx}{\int_{\R^N}\varphi(\varepsilon_q x)w_q^2(x)dx}\leq C e^{-\frac{\mu}{\varepsilon_q} } \ \text{ as } q\nearrow q^*,
\end{equation}
\begin{align}\label{4.4}
\int_{\mathbb{R}^N}V(x)\tilde u_q^2(x)dx&=A_q^2\int_{\mathbb{R}^N}V(\varepsilon
x+x_0)\varphi (\varepsilon_q x)w_q^2dx\nonumber\\
&\rightarrow V(x_0)\int_{\mathbb{R}^N}w_0^2dx=V(x_0)\ \text{ as } q\nearrow q^*,
\end{align}
and
\begin{align}
\int_{\R^N}|\tilde u_q|^{q+2}dx&=\varepsilon_q^{-\frac{Nq}{2}}A_q^{q+2}\int_{\R^N}\varphi^{q+2}(\varepsilon_q x)|w_q|^{q+2}dx=\varepsilon_q^{-\frac{Nq}{2}}\int_{\R^N}|w_q|^{q+2}dx+O\big(e^{-\frac{\mu}{\varepsilon_q}} \big)\nonumber\\
&=\int_{\R^N}| u_q|^{q+2}dx+O\big(e^{-\frac{\mu}{\varepsilon_q}} \big)\ \text{ as } q\nearrow q^*.
\end{align}
Similar to the above argument, one can also prove that
\begin{align}
\int_{\R^N}|\nabla\tilde u_q^2|^2dx
&=\int_{\R^N}|\nabla  u_q^2|^2dx+O\big(e^{-\frac{\mu}{\varepsilon_q}} \big)\ \text{ as } q\nearrow q^*,
\end{align}
and
\begin{align}
\int_{\R^N}|\nabla\tilde u_q|^2dx
&=\int_{\R^N}|\nabla  u_q|^2dx+O\big(e^{-\frac{\mu}{\varepsilon_q}} \big)\ \text{ as } q\nearrow q^*.
\end{align}
Therefore, choosing $x_0\in \R^N$ such that $V(x_0)=0$,  we then deduce from the above estimates that
\begin{align*}
0&\leq d_a(q)-\tilde{d}_a(q)\leq E_q^a(\tilde u_q(x))-\tilde{E}_q^a(u_q(x))\\
&= \tilde E_q^a(\tilde u_q(x))-\tilde{E}_q^a(u_q(x))+\frac{1}{2}\int_{\mathbb{R}^N}V(x)\tilde u_q^2(x)dx\\
&=\frac{1}{2}V(x_0) +O\big(e^{-\frac{\mu}{\varepsilon_q}}\big)+o(1)\rightarrow 0 \ \text{ as } q\nearrow q^*.
\end{align*}
Moreover, if $\bar{u}_q$ is a nonnegative minimizer of $d_a(q)$. Then
\begin{equation*}
\int_{\mathbb{R}^N}V(x)\bar{u}_q^2dx=d_a(q)-\tilde{E}_q^a(\bar{u}_q)\leq d_a(q)-\tilde{d}_a(q)\rightarrow 0\ \text{ as } \ q\nearrow q^*.
\end{equation*}

\end{proof}

\noindent\textbf{Proof of Theorem \ref{th1.3}:}
Now we still denote $\bar{u}_q$ be a nonnegative minimizer of $d_a(q)$.  Applying Lemma \ref{l4.1}, one can check that all the conclusions in Lemma \ref{le4.2}  also holds for $\bar{u}_q$, i.e.,
\begin{equation}\label{4.44}
\int_{\mathbb{R}^N}|\nabla \bar{u}_q^2|^2dx\approx \frac{4a}{q+2}\int_{\mathbb{R}^N}\bar{u}_q^{q+2}dx\approx \left(\frac{4aq}{q^*\lambda_q a_q(q+2)}\right)^{\frac{q^*}{q*-q}}=t_q
\end{equation}
and
\begin{equation}\label{4.5}
\frac{\int_{\mathbb{R}^N}|\nabla \bar{u}_q|^2dx}{\int_{\mathbb{R}^N}|\nabla \bar{u}_q^2|^2dx}\rightarrow 0 \ \text{ as } q\nearrow q^*.
\end{equation}
Moreover, similar to  (\ref{3.16}), one can prove that there exists $\{y_{\varepsilon_q}\}\subset\R^N$ such that the scaling
$$\bar w_q(x):= \varepsilon_q^{\frac{N}{2}}\bar u_q(\varepsilon_q
x+\varepsilon_qy_{\varepsilon_q})$$
satisfies
\begin{equation*}
\liminf_{q\nearrow q^*}\int_{B_R(0)} |\bar w_q|^{2}dx>\eta>0.
\end{equation*}
 Then, repeating the proof of Theorem \ref{th4.1}, we  can prove that
$$\bar w_q^2\rightarrow  w^2_0:=\frac{\lambda^N}{|v_{q^*}|_{L^1}}v_{q^*}\Big(\lambda |x-x_0|\Big) \ \text{ in $\mathcal {D}^{2,1}(\mathbb{R}^N)$} \text{ with }\lambda=\left(\frac{|v_{q^*}|_{L^1}}{N}\right)^\frac{1}{N+2}.$$
 Moreover, by (\ref{4.3}) we see that
 \begin{equation*}
\int_{\mathbb{R}^N}V(x)\bar{u}_q^2dx=\int_{\R^N}V(\varepsilon_qx+\varepsilon_qy_{\varepsilon_q})\bar w_q(x)dx\to0 \ \text{ as }\ q\nearrow q^*.
\end{equation*}
This further indicates that the sequence $\{\varepsilon_q y_{\varepsilon_q}\}$ satisfies
$$\varepsilon_q y_{\varepsilon_q}\rightarrow A=\{x:V(x)=0\} \text{ as }q\nearrow q^*.$$
The proof of Theorem \ref{th1.3} is complete.
\qed

\end{document}